\documentclass[final]{siamltex}

\usepackage{amsfonts,amsmath,graphicx}
\usepackage{pmat,multirow,slashbox}
\SetSymbolFont{letters}{bold}{OML}{cmm}{b}{it}
\SetSymbolFont{operators}{bold}{OT1}{cmr}{bx}{n}

\def \f{{\bf f}}

\def \q{{\bf q}}
\def \u{{\bf u}}\def \x{{\bf x}}
\def \y{{\bf y}}\def \z{{\bf z}}
\def \dd{{\text{\! d}}}\def \R{\mathbb{R}}
\def \V{{\cal V}}\def \Q{{\cal Q}}
\newcommand{\suchthat}{\;\ifnum\currentgrouptype=16 \middle\fi|\;}
\newtheorem{remark}{Remark}[section]
\newcommand{\norm}[2]{\ensuremath{\| #1 \|_{#2}}}
\newcommand{\seq}[1]{\ensuremath{\left\{ #1 \right\}}}
\title{Constraint interface preconditioning for\\ topology optimization problems}

\author{M. Ko\v{c}vara\footnotemark[2]
\and D. Loghin\footnotemark[3]
\and J. Turner\footnotemark[3]}

\begin{document}
\maketitle
\renewcommand{\thefootnote}{\fnsymbol{footnote}}

\footnotetext[2]{School of
    Mathematics, University of Birmingham, Edgbaston, Birmingham B15 2TT, UK,
    and Institute of Information Theory
and Automation, Czech Academy of Sciences, Pod
vod\'arenskou v\v{e}\v{z}\'{\i}~4, 18208 Praha 8, Czech Republic. The
work of this author has been partly supported by the EU FP7 project
AMAZE and by grant A100750802 of the Czech Academy of Sciences}
\footnotetext[3]{School of
    Mathematics, University of Birmingham, Edgbaston, Birmingham B15 2TT, UK}

\renewcommand{\thefootnote}{\arabic{footnote}}

\begin{abstract}
The discretization of constrained nonlinear optimization problems
arising in the field of topology optimization yields algebraic systems
which are challenging to solve in practice, due to pathological
ill-conditioning, strong nonlinearity and size. In this work we propose
a methodology which brings together existing fast algorithms, namely,
interior-point for the optimization problem and a novel substructuring
domain decomposition method for the ensuing large-scale linear systems.
The main contribution is the choice of interface preconditioner which
allows for the acceleration of the domain decomposition method, leading
to performance independent of problem size.
\end{abstract}



\pagestyle{myheadings} \thispagestyle{plain} \markboth{Constraint
interface preconditioning for topology optimization}{Constraint
interface preconditioning for topology optimization}

\section{Introduction}
The aim of topology optimization is to determine the optimal
distribution of a certain amount of material within a prescribed
value in order to minimise the strain energy (or compliance) of a given
structure. The distinguishing feature separating this approach from
shape optimization involves the introduction of new boundaries,
allowing for the consideration of a broader range of feasible
solutions. The main problem when pursuing such an
approach is that a large number of design variables are required in the
discrete formulation in order to maintain the quality of the contours
in the final design. Solutions are typically obtained through the use
of iterative optimization techniques, requiring repeated
discretizations via the finite element method, corresponding to a
sequence of linearized problems. As a result, even problems resulting
from using relatively coarse discretization parameters can be
computationally demanding. Our aim in this paper is to introduce
solution methods adapted to the complex nature of this class of
problems.

Attempts to alleviate such difficulties can involve the application of
a faster finite element solver, or the use of efficient discretisation
techniques \cite{DeRose00}. Standard approaches based around Picard
iterations target the ill-conditioned equilibrium equations, from which
the bulk of computational effort resides. In \cite{Wang07}, MINRES
coupled with recycling is explored based on the observation that the
densities are only expected to undergo minor changes after a relatively
small number of iterative steps. The ill-conditioning is dealt with
through a preconditioning strategy involving both rescaling and an
incomplete Cholesky decomposition.

In terms of parallel computing, the application of the preconditioned
conjugate gradient (PCG) method coupled with Jacobi preconditioning has
been considered in a number of references, including
\cite{Borrvall01,Evgrafov08,Kim04,Mahdavi06,Vemaganti05}. Two
additional approaches are considered in \cite{Vemaganti05}, namely
preconditioning based on an ILU factorisation, as well as condensation
through substructuring coupled with a diagonal preconditioner for the
resulting interface problem.

Alternatively, primal-dual Newton methods can be considered, with
particular focus on interior point approaches. Examples illustrating
the application of such approaches for solving large scale topology
optimization problems can be found in \cite{Ben00,Hoppe04,Maar00}.
The KKT conditions from the resulting nonlinear equality constrained
optimization problem are then solved using Newton's method. Evidently,
for large scale problems, obtaining solutions to the resulting system
of equations will become expensive and even prohibitive in certain
cases. In \cite{Maar00}, Maar and Schultz applied multigrid to the
resulting system, and from their results were able to witness an
approximately linear overall complexity with respect to the number of
unknowns used in the problem.

In this paper, we propose to apply domain decomposition to the
resulting Newton system described above, which will then be solved
using GMRES coupled with an appropriate preconditioning strategy. An
important component within our preconditioner is based on targeting the
resulting interface problem, which will be achieved through the
consideration of an appropriate fractional Sobolev norm. Our paper will
conclude by illustrating that results can be obtained without
dependence on the chosen mesh parameter.

\section{Problem description}
Consider a material occupying an open and connected domain $\Omega
\subset \mathbb{R}^2$ with Lipschitz boundary $\partial \Omega =
\partial \Omega_D \cup \partial \Omega_N$. We assume that the
elasticity tensor for the material can be modeled as
$E(\x)=\rho(\x)^\mu\bar{E}$, where $\rho(\x)$ is the density,
$\mu\in[1,3]$ and $\bar{E}$ is a prescribed constant tensor. This
expression describes a common approximation of the material properties
known as the Solid Isotropic Material with Penalisation (SIMP) model.
In this paper we will be interested in the Variable Thickness Sheet
(VTS) problem, corresponding to the choice $\mu=1$, for which
existence of solutions holds (see, e.g.,
\cite[p.272--274]{Bendsoe03}). We assume that the material is clamped
(i.e., we assume Dirichlet conditions) 
along the subset $\partial\Omega_D$ of the boundary and that both body forces $f:\Omega \rightarrow
\mathbb{R}^2$ and boundary tractions $g:\partial \Omega_N \rightarrow
\mathbb{R}^2$ act on the material, resulting in a displacement $u$
which satisfies the following equilibrium equations in $\Omega$:
$$-\textrm{div}\sigma(u)=f,$$
where, by Hooke's law, the stress tensor is $\sigma(u)=E:\varepsilon(u)$. 

We employ next
a weak formulation in order to define our topology optimization
problem. The natural spaces arising in this context are the standard
spaces $L^2(\Omega),
H^1(\Omega)$ and
$$H^1_D(\Omega)=\seq{v\in H^1(\Omega): v_{\mid\partial\Omega_D}=0}.$$
For the description of the preconditioners associated with
a domain decomposition approach with corresponding interface $\Gamma$
between subdomains, we will require the
fractional Sobolev space 
$$\Lambda:=[H^1_D(\Gamma), L^2(\Gamma)]_{1/2}.$$
which is an interpolation space of index 1/2 (see \cite{Lions70} for
details).
\subsection{Weak formulation} Let $\V=[H^1_D(\Omega)]^2$ and let
$a_{\rho}(\cdot,\cdot):\V\times \V\rightarrow \mathbb{R}$ and
$F(\cdot):\V\rightarrow \mathbb{R}$ be defined below
\[
a_{\rho}(w,v) =\int_{\Omega} \rho(\x)\left( \varepsilon(w):\bar{E}:\varepsilon(v)\right)\dd\x,
\]
where $\varepsilon(v)$ is the strain tensor corresponding to a
displacement $v$ and
\[
F(v) :=\int_{\Omega} f \cdot v \dd\x+\int_{\partial \Omega_N} g\cdot v\dd s.
\]
We seek $u\in\V$ such that for all $v\in\V$
$$a_\rho(u,v)=F(v).$$

In order to formulate our topology optimization problem, we define the
following admissible set for our design variable $\rho$:
\[
\Q^c=\seq{\rho\in \Q\equiv L^\infty(\Omega) \suchthat 0< \underline{\rho} \leq
  \rho(\x) \leq \overline{\rho} \text{ a.e.~in } \Omega, \int_\Omega\rho(\x)\dd\x=\mathcal{M}_{\Omega}},
\]
where $\mathcal{M}_{\Omega}$ denotes the amount of material
available, and $\overline{\rho}$ and $\underline{\rho}$ denote upper
and lower limits on the density function, respectively.

Consider now the following nonlinear minimization problem:
\begin{align}
   &\min_{(u,\rho)\in \V\times \Q^c} &&\hspace{-2cm} \frac{1}{2}a_\rho(u,u)\quad
   \left[=\frac{1}{2}F(u)\right] &&  \label{WFTOrho}\\
   &\textrm{subject to:} &&\hspace{-2cm} a_{\rho}(u,v)=F(v)
   &\hspace{-4cm} \forall v\in \V.&    \label{Elasticity}
\end{align}
The state variable corresponds to the displacement of the material $u$,
while the design variable is the density function
$\rho(\x):\Omega\rightarrow\mathbb{R}_+$.

Let now $\V_h, \Q_h^c$ denote finite-dimensional subspaces of $\V$ and
$\Q^c$, respectively and consider the following discrete weak formulation of our minimization problem:
\begin{align}
   &\min_{(u_h,\rho_h)\in \V_h\times \Q_h^c} &&\hspace{-2cm} \frac{1}{2}a_\rho(u_h,u_h)\quad \left[=\frac{1}{2}F(u_h)\right] &&  \label{WFTOrhoh}\\
   &\textrm{subject to:} &&\hspace{-2cm} a_{\rho}(u_h,v_h)=F(v_h)  &\hspace{-4cm} \forall v_h\in \V_h
.&    \label{Elasticityh}
\end{align}
Using finite element bases for $\V_h$ and $\Q_h$, the minimization
problem (\ref{WFTOrhoh}--\ref{Elasticityh}) yields the following
discrete constrained minimization problem:
\begin{align}
   &\min_{(\u,\pmb{\rho})\in\mathbb{R}^n\times\mathbb{R}^m} &&
   \hspace{-2cm} \frac{1}{2}\u^TA(\pmb{\rho})\u \quad\left[=\frac{1}{2}\f^T\u\right] &&  \label{WFTOrhod}\\
   &\textrm{subject to:} &&\hspace{-1.5cm} A(\pmb{\rho})\u=\f &
&    \label{Elasticityd}\\
&&&\hspace{-1.5cm} \q^T\pmb{\rho}= \q^T{\bf 1},&
&    \label{massconstraint}\\
&& \hspace{-2cm} \underline{\rho}{\bf 1} \leq \pmb{\rho} \leq
\overline{\rho}{\bf 1},&\label{densityconstraint}
\end{align}
where $q_i:=[\q]_i=\mid \!T_i\!\mid$. Working with a piecewise
constant approximation $\rho_h$, we can write the assembly process
of the stiffness matrix as follows:
$$A(\pmb{\rho})=\sum_{i=1}^m \pmb{\rho}_i A_i$$
where $A_i$ is global representation of the elemental matrix
corresponding to simplex $T_i$. This will allow for a certain
simplified expression for the Jacobian matrix when considering the
first-order conditions for our minimization problem.
\section{Interior point method} Traditional and popular approaches for
solving optimization problems such as
(\ref{WFTOrhod}--\ref{densityconstraint}) involve separate treatment of
both the design objective and the equilibrium equations. Typically, for
an initial given design the stiffness matrix is assembled and used to
solve the equilibrium equations for the displacement $\mathbf{u}$. This
$\mathbf{u}$ is then used to obtain an appropriate update to the design
variables, which is then checked for suitability based on previous
values. If an appropriate solution has yet to be found, the process is
repeated. Typical approaches used to obtain an update to the design
variables include both the Optimality Criteria (OC) method
\cite[p.308]{Bendsoe03} and the Method of Moving Asymptotes (MMA)
\cite{Svanberg02}, usually followed by both sensitivity and filtering
analysis to cater for the general SIMP setting \cite{Sigmund01}.

More recently, fully-coupled approaches have been receiving
considerable attention within the PDE constrained optimization
community. In these approaches, all constraints are included and no
sub-problems are being solved separately. An important feature within
these methods is that the equilibrium equations are embedded within the
optimization routine, allowing for the simultaneous treatment of all
constraints within the problem. Examples highlighting the benefits, and
in particular the savings in computational time for such methods, can
be found in a number of sources, including
\cite{Biros05,Birosa05,Hoppe04}. We describe below an interior point
method as applied to the topology optimization problem introduced
in (\ref{WFTOrhod}-\ref{densityconstraint}).

Interior point methods are used to solve both convex linear and
nonlinear optimization problems iteratively by considering updates
confined to the feasible region (cf.\ \cite{Byrd99, El96, Wright97}).
As well as obtaining solutions in a polynomial time, these methods have
been used to determine solutions to previously intractable problems,
meaning that they are useful from both a theoretical as well as a
practical view point. We begin by rewriting the formulation
(\ref{WFTOrhod}) slightly to incorporate the inequality constraints
within the objective function. This will be achieved through the use of
logarithmic barrier terms as illustrated below:
\begin{align}
   &\min_{(\u,\pmb{\rho})\in\mathbb{R}^n\times\mathbb{R}^m} &&
   \hspace{-1cm} \frac{1}{2}\u^TA(\pmb{\rho})\u- r \displaystyle\sum_{i=1}^m
   \log({\rho}_i - \underline{\rho}) - s \displaystyle\sum_{i=1}^m
   \log(\overline{\rho} - {\rho}_i)  && \label{DWFTOlog}\\
   &\textrm{subject to:} &&\hspace{-1cm} A(\pmb{\rho})\u=\f &
&    \label{Elasticityd1}\\
&&&\hspace{-1cm} \q^T\pmb{\rho}= \q^T{\bf 1}&
&    \label{massconstraint1}
\end{align}
with $r,s>0$ and where $[\pmb{\rho}]_i=\rho_i$.

The Lagrangian associated with the problem (\ref{DWFTOlog}--\ref{massconstraint1}) is
\begin{align}
\label{LagBar}
\mathcal{L}^{(r,s)} \left( \mathbf{v}, \lambda, \mathbf{u}, \boldsymbol{\rho} \right)
{ \mathrel{\,\vcenter{\baselineskip0.5ex \lineskiplimit0pt\hbox{\small.}\hbox{\small.}}}=\, }\frac{1}{2}\u^TA(\pmb{\rho})\u - r \displaystyle\sum_{i=1}^m
 \log(\rho_i - \underline{\rho}) - s \displaystyle\sum_{i=1}^m
 \log(\overline{\rho} - \rho_i)\ \ \ \\ +\mathbf{v}^T \left( \mathbf{f} -A(\boldsymbol{\rho}) \mathbf{u} \right) +\lambda \left(\q^T{\bf 1}-\q^T\pmb{\rho}\right). \nonumber
\end{align}

The stationary points are defined by setting to zero the relevant
partial derivatives of the Lagrangian (\ref{LagBar}):
\begin{align}
\nabla_{\mathbf{v}} \mathcal{L} &=  \mathbf{f} - A(\boldsymbol{\rho}) \mathbf{u}  = \mathbf{0}, \label{SP1} \\
\nabla_{\lambda} \mathcal{L} &=\q^T \pmb{\rho} -\q^T{\bf 1}= 0, \label{SP2} \\
\nabla_{\mathbf{u}} \mathcal{L} &= A(\boldsymbol{\rho}) \mathbf{u} - A(\boldsymbol{\rho}) \mathbf{v} = \mathbf{0}, \hspace{0.2in}  \label{SP3} \\
\nabla_{\boldsymbol{\rho}} \mathcal{L} &= \frac{1}{2}  B(\u)^T\u-
                                         B(\u)^T\mathbf{v} - \lambda \q - r X^{-1} \mathbf{1} + s \tilde{X}^{-1} \mathbf{1}  = \mathbf{0}. \label{SP4}
\end{align}
In the above, $B(\mathbf{u}){ \mathrel{\,\vcenter{\baselineskip0.5ex \lineskiplimit0pt\hbox{\small.}\hbox{\small.}}}=\, } \left[ A_1 \mathbf{u}, A_2 \mathbf{u}, \dots, A_m
\mathbf{u} \right] \in \mathbb{R}^{n \times m}$ and
\[ X{ \mathrel{\,\vcenter{\baselineskip0.5ex \lineskiplimit0pt\hbox{\small.}\hbox{\small.}}}=\, }\diag(\pmb{\rho}- \underline{\rho}{\bf 1}) \hspace{0.5in}
\tilde{X} { \mathrel{\,\vcenter{\baselineskip0.5ex \lineskiplimit0pt\hbox{\small.}\hbox{\small.}}}=\, } \diag(\overline{\rho}{\bf 1} - \pmb{\rho}).\]
It is important to note that the condition number of the Hessian of the
Lagrangian may pose an issue for densities close to either
$\underline{\rho}$ or $\overline{\rho}$ as both $r$ and $s$ tend to
zero. To alleviate this issue, auxiliary non-negative variables
$\boldsymbol{\phi}$ and $\boldsymbol{\psi}$ are introduced in the
following way
\begin{equation}
\label{PertComp}
\boldsymbol{\phi} = \boldsymbol{\phi}^{(r,s)} { \mathrel{\,\vcenter{\baselineskip0.5ex \lineskiplimit0pt\hbox{\small.}\hbox{\small.}}}=\, } r X^{-1} \mathbf{1}, \hspace{0.25in} \textrm{and} \hspace{0.25in} \boldsymbol{\psi} = \boldsymbol{\psi}^{(r,s)} { \mathrel{\,\vcenter{\baselineskip0.5ex \lineskiplimit0pt\hbox{\small.}\hbox{\small.}}}=\, }s \tilde{X}^{-1} \mathbf{1}.
\end{equation}
Using (\ref{PertComp}), we see that
\[
\Phi X \mathbf{1} = r \mathbf{1}, \hspace{0.25in} \textrm{and} \hspace{0.25in} \Psi \tilde{X} \mathbf{1} = s \mathbf{1}, \label{SP5}
\]
where $\Phi {\mathrel{\vcenter{\baselineskip0.5ex
\lineskiplimit0pt\hbox{\small.}\hbox{\small.}}}=}
\diag(\boldsymbol{\phi})$ and $\Psi
{\mathrel{\vcenter{\baselineskip0.5ex
\lineskiplimit0pt\hbox{\small.}\hbox{\small.}}}=}
\diag(\boldsymbol{\psi})$. Through this substitution, and the
elimination of the Lagrange multiplier $\mathbf{v}$ (which, by
(\ref{SP1}) and (\ref{SP3}), is equal to $-\mathbf{u}$), the first
order optimality conditions can be written as
\[
\mathcal{R} \left( \mathbf{u}, \lambda, \boldsymbol{\rho}, \boldsymbol{\phi}, \boldsymbol{\psi}  \right){ \mathrel{\,\vcenter{\baselineskip0.5ex \lineskiplimit0pt\hbox{\small.}\hbox{\small.}}}=\, } \nabla \mathcal{L}^{(r,s)} \left( \mathbf{u}, \lambda, \boldsymbol{\rho}, \boldsymbol{\phi}, \boldsymbol{\psi}  \right) = \left( \begin{array}{c}
\mathbf{f} - A(\boldsymbol{\rho}) \mathbf{u} \\
\q^T\pmb{\rho}-\q^T{\bf 1} \\
\frac{1}{2}  B(\u)^T\u + \lambda \q + \boldsymbol{\phi} - \boldsymbol{\psi}\\
r \mathbf{1} - \Phi X \mathbf{1}\\
s \mathbf{1} - \Psi \tilde{X} \mathbf{1} \end{array} \right) =
\left( \begin{array}{c}
\mathbf{0} \\
0 \\
\mathbf{0}\\
\mathbf{0}\\
\mathbf{0}\end{array} \right).
\]
By setting $\y= \left( \mathbf{u}, \lambda, \boldsymbol{\rho},
  \boldsymbol{\phi}, \boldsymbol{\psi} \right)^T$, Newton's method
applied to the above nonlinear optimality conditions has the following form
\begin{equation}
\label{NK1}
J\left( \y^{k-1} \right) \Delta \y^k= \mathcal{R}\left( \y^{k-1} \right),
\end{equation}
where the Jacobian matrix $J\left( \y\right)$ is given below
\begin{equation}
\label{JacobianFull}
J(\y) =
  \begin{bmatrix} A
(\boldsymbol{\rho}) & & B(\mathbf{u}) & & \\
    & & \q^T & & \\
    B(\mathbf{u})^T & \q & & I_{m} & -I_{m}\\
    &  & \Phi & X & \\
     & & -\Psi &  & \widetilde{X} \end{bmatrix}.
    \end{equation}
\begin{remark}\label{Jac_assumption}
  In the following, we will assume that problems of the form
  \textup{(\ref{DWFTOlog}--\ref{massconstraint1})} yield Jacobian matrices
  $J(\y)$ of the form \textup{(\ref{JacobianFull})}, which are non-singular for
  $\y$ in a neighbourhood of the solution.  This property will be
  assumed to hold for both Dirichlet and mixed boundary conditions,
  due to the well-posedness of problem
  \textup{(\ref{WFTOrho}--\ref{Elasticity})} for both mixed Dirichlet-Neumann
  and Dirichlet-only boundary conditions.
\end{remark}

Despite $J$ being both
nonsymmetric and indefinite, its condition number is expected to be
bounded under reduction of the barrier parameters $r$ and $s$.
Therefore, it is important to consider appropriate strategies for
obtaining an accurate update through (\ref{NK1}). Work by Forsgren,
Gill and Shinnerl \cite{Forsgren96} uses the diagonal structure of both
$\Phi$ and $\Psi$ to transform $J$ into a symmetric matrix. Another
possibility is to consider appropriate techniques to condense the
matrix $J$ via block elimination or using a Schur complement
approach. This approach is known to lead to ill-conditioning; however,
the effect on the accuracy of the resulting solution can be benign, as discussed by Wright
in \cite{Wright98}. Therefore, block elimination techniques
remain a practical option, particularly for the situation where the original matrix $J$ is
large. The drawback in this case is the loss of sparsity and the lack
of obvious (and efficient) preconditioners. For this reason, we look to exploit
the sparsity present in the original unreduced system by using an
iterative method coupled with an appropriate preconditioning strategy,
with the original matrix shown to exhibit favourable spectral
properties as shown in recent work by Greif et.\ al.\ \cite{Greif14}.
The preconditioner employed for this work will be based on a
decomposition of the domain $\Omega$ into subdomains, described in detail
from the next section onwards. In particular, we provide a description of
an interface preconditioning technique which was first introduced in
\cite{Arioli09} for solving scalar elliptic problems, and which is now adapted to the case of our constrained
PDE problem.

\section{Domain decomposition}
\label{DomDecSect}
In order to formulate a decomposition of our problem that is suitable
for a parallel environment, we need to ensure that aside from the
physical decomposition of the domain, the resulting subproblems are
well-posed. It turns out that we can achieve this through a simple
reformulation of our minimization problem, which targets the mass
constraint (\ref{massconstraint1}).
\subsection{Standard definitions and notation}
Consider a subdivision of $\Omega$ into $N$ non-overlapping subdomains
$\Omega_k$ with boundaries $\partial\Omega_k$ such that
$$\bar{\Omega}=\bigcup_{k=1}^N\bar{\Omega}_k,~~~\Omega_k\cap\Omega_j\equiv\emptyset ~~
(k\neq j).$$ We denote the resulting interface by $\Gamma$: $$\Gamma { \mathrel{\,\vcenter{\baselineskip0.5ex \lineskiplimit0pt\hbox{\small.}\hbox{\small.}}}=\, }
\bigcup^{N}_{k=1} \Gamma_k,~~~~\Gamma_k { \mathrel{\,\vcenter{\baselineskip0.5ex \lineskiplimit0pt\hbox{\small.}\hbox{\small.}}}=\, }\partial\Omega_k \backslash
\partial\Omega.$$ For each $k$, we define the \emph{nodal index set}
$\nu_k$ to be the set of nodes strictly contained in $\Omega_k$ and the
\emph{simplex index set} $\tau_k$ to be the set of indices of all
simplices contained in $\Omega_k$. We further define $m_k{ \mathrel{\,\vcenter{\baselineskip0.5ex \lineskiplimit0pt\hbox{\small.}\hbox{\small.}}}=\, }\mid\!\tau_k\!\mid$.
\subsection{Problem reformulation}
One of the obstacles in decomposing the VTS problem
(\ref{WFTOrhod}--\ref{densityconstraint}), or equivalently the interior
point formulation (\ref{DWFTOlog}--\ref{massconstraint1}), is the
global mass constraint (\ref{massconstraint}), or
(\ref{massconstraint1}), respectively. A suitable way to 'decompose'
this constraint is provided by the following equivalent formulation:
$$\q_k^T\pmb{\rho}_k=\mu_k,~~~(k=1,\ldots,N),~~~\sum_{k=1}^N\mu_k=\q^T{\bf 1},$$
where $\mathbf{q}_k$ and $\boldsymbol{\rho}_k$ represent the respective sub-vectors of $\mathbf{q}$ and $\boldsymbol{\rho}$ corresponding to the index set $\tau_k$. We note here that this approach introduces $N$ additional
unknowns $\left\{\mu_k, k=1,\ldots,N\right\}$, together with $N$
additional constraints. Note also that given the definition of $\q$,
$\mu_k$ represents the mass of subdomain $\Omega_k$.

The modified formulation corresponding to
(\ref{DWFTOlog}--\ref{massconstraint1}) is included below.
\begin{align}
  &\min_{(\u,\pmb{\rho})\in\mathbb{R}^n\times\mathbb{R}^m} &&\hspace{-1cm} \frac{1}{2} \u^TA(\pmb{\rho})\u- r \displaystyle\sum_{i=1}^m \log({\rho}_i - \underline{\rho}) - s \displaystyle\sum_{i=1}^m \log(\overline{\rho} - \rho_i),  && \label{DWFTOlog1}\\
   &\textrm{subject to:} &&\hspace{-1cm} A(\pmb{\rho})\u=\f ,&
&    \label{Elasticityd2}\\
&&&\hspace{-1cm} \q_k^T\pmb{\rho}_k= \mu_k,&\hskip-10em k=1,\ldots,N,
&    \label{massconstraintka}\\
&&&\hspace{-1cm} \sum_{k=1}^N\mu_k=\q^T{\bf 1}.&&\label{massconstraintkb}
\end{align}
The above problem is equivalent to minimization problem (\ref{WFTOrhod}-\ref{densityconstraint}) and
hence is well-posed.

Following the same procedure of differentiating the corresponding
Lagrangian function, one can derive the first-order optimality
conditions and set up Newton's iteration in the form (\ref{NK1}). The
resulting Jacobian matrix, also denoted by $J$, now has the form
\begin{equation}
\label{JacobianE}
J=
  \begin{bmatrix} A & & B & & & & \\
   & & Q & & & {-I_{N}} & \\
    B^T & Q^T & & I_m & -I_{m}& & \\
   & & \Phi & X& & & \\
    & & -\Psi & & \widetilde{X} & & \\
   & {-I_{N}} && && & {\mathbf{1}_{N}} \\
   & &  & & & {\mathbf{1}_{N}^T} & \end{bmatrix}.
\end{equation}
where $Q\in\R^{N\times m}$ is described below
\[
Q_{kj}:=
\begin{cases}
q_j & \text{if} \hspace{0.2in} j \in \tau_k, \\
0 & \text{otherwise}.
\end{cases}
\]
Using a node ordering comprising nodes interior to the subdomains
$\Omega_k$ followed by nodes on the interface $\Gamma$, the Jacobian
will have the following permuted block form (with subscripts $I$ and
$\Gamma$ indicating this ordering)
\begin{equation}
\label{ExpandedJ}
\begin{pmat}[{|}]
\renewcommand{\arraystretch}{1.5} J_{I I} & J_{I \Gamma} \cr\- J_{\Gamma I} & J_{\Gamma \Gamma} \cr \end{pmat} =
\begin{pmat}[{....|}]
\def\arraystretch{0.5}
A_{II} & & B_I & & & A_{I \Gamma} & & \cr
& & Q && &  & -I_{N} & \cr
B_I^T & Q^T & & I_{m} & -I_{m}& B_{\Gamma}^T &  & \cr
& & \Phi_m & X_{m} & & & &  \cr
 & & -\Psi_m &  & \widetilde{X}_{m} &  & & \cr \-A_{\Gamma I} &  & B_{\Gamma}  &  &   & A_{\Gamma \Gamma} & &  \cr
 & -I_{N}  &  &  &  & &  & \mathbf{1}_{N} \cr
 & &  & &  & &  \mathbf{1}^T_N &\cr
\end{pmat}.
\end{equation}
Under a further permutation which lists all the unknowns corresponding
to subdomains $\Omega_k$, for each $k=1,\ldots,N$, it can be seen that 
it has the block-diagonal structure $J_{II} :=\bigoplus_{k=1}^N J^k_{II}$.
Each block $J_{II}^k$ represents the instance
of the Jacobian for a minimization problem posed over the subdomain
$\Omega_k$ with Dirichlet boundary conditions on $\partial\Omega_k$ and
which has the familiar algebraic form (\ref{JacobianFull}):
$$J_{II}^k{ \mathrel{\,\vcenter{\baselineskip0.5ex \lineskiplimit0pt\hbox{\small.}\hbox{\small.}}}=\, }\begin{bmatrix}
  A_k & & B_k & & \\
    & & \mathbf{q}_{m_k}^T & & \\
    B_k^T & \mathbf{q}_{m_k} & & I_{m_k} & -I_{m_k}
\\
    &  & \Phi_{m_k} & X_{m_k} & \\
     & & - \Psi_{m_k} &  & \widetilde{X}_{m_k} \end{bmatrix}.
    $$
Here, $A_k,B_k$ are the counterparts of $A,B$
assembled on the interior of $\Omega_k$ and with Dirichlet
boundary conditions applied. By the well-posedness of problems of
the form (\ref{DWFTOlog}--\ref{massconstraint1}) for the case $\partial\Omega_N=\emptyset$, we
conclude that the associated first-order optimality conditions are
well-posed and yield a Jacobian matrix which is non-singular
(cf. Remark \ref{Jac_assumption}).

The direct-sum structure of $J_{II}$ was achieved through the
reformulation of the mass constraint and clearly allows for a
parallel implementation of the inverse of $J_{II}$, which will be a
key building block in the solution algorithm described in the next
section. It is also an interesting feature in itself to be able to
decompose the global Jacobian into local Jacobians associated with
similar local minimization problems. One consequence is that the
local problems inherit the well-posedness associated with the global
problem; the resulting decomposition is hence useful, with the local
problems invertible independently of each other.

\subsection{A Dirichlet-Dirichlet approach} At each step of the Newton
iteration we need to solve linear systems of the form
$$J\z={\bf r},$$
where $J$ is the Jacobian matrix (\ref{ExpandedJ}) corresponding to the
first order conditions for the reformulated problem
(\ref{DWFTOlog1}--\ref{massconstraintkb}). Due to the size and
structure of $J$, we seek to solve such systems using an iterative
solver with a suitable preconditioner. A proven candidate
is the block upper-triangular matrix below employed as a right preconditioner
\begin{equation}
\label{precP}
P=\begin{bmatrix}{J_{II}}&{J_{I\Gamma}}\\{}&\tilde{S}\end{bmatrix},
\end{equation}
where $\tilde{S}$ is an approximation to the Schur complement matrix
$$S=J_{\Gamma\Gamma}-J_{\Gamma I}J_{II}^{-1}J_{I\Gamma}.$$
Due to the direct-sum property of $J_{II}$, this block approach can
be classified in domain decomposition terminology as a non-overlapping Dirichlet-Dirichlet
procedure, where each subdomain block represents a Jacobian matrix
arising in some topology optimisation sub-problem posed on $\Omega_k$ where the material
is clamped on $\Gamma_k$.
Moreover, at each Newton iteration, the Schur complement $S$ can be seen as the
finite element discretization of a generalised Steklov-Poincar\'{e}
operator corresponding to the interface problem generated by the
decomposition. We remark here again that this approach is only
available via the re-formulation described above

The preconditioner inverse can be written as the product of three matrices:
\begin{equation}
\label{precPinv}
P^{-1} = \begin{bmatrix}
J_{II}^{-1} & 0 \\ 0 & I_{\Gamma \Gamma}
\end{bmatrix}\begin{bmatrix}
I_{II} & -J_{I \Gamma} \\ 0 & I_{\Gamma \Gamma}
\end{bmatrix}\begin{bmatrix}
I_{II} & 0 \\ 0 & \tilde{S}^{-1}
\end{bmatrix}.
\end{equation}
The application of $P^{-1}$ would initially
involve the action of $\tilde{S}^{-1}$ on the skeleton problem
corresponding to the interface $\Gamma$. 
Next, a
boundary-to-domain update would be applied through $J_{I \Gamma}$,
before applying the inversion in parallel of $J_{II}$ on subdomains. With the exception of
$\tilde{S}$, the potential for parallelism in (\ref{precP}), or equivalently
in (\ref{precPinv}), is evident. Therefore, the task is to seek an
appropriate representation to $\tilde{S}$ so that the preconditioner
can be assembled, stored and applied in an efficient manner.
This is discussed in detail in the next section.
\section{Constraint interface preconditioners}
The Schur complement has the
following block $3 \times 3$ structure
\begin{equation}
S=
\begin{bmatrix}
{S_{11}} & {S_{12}} & {\bf 0} \\
{S_{12}^T} & {S_{22}} & {\mathbf{1}_N} \\
{\bf 0}^T & {\mathbf{1}^T_N}  & 0
\end{bmatrix},\label{schur}
\end{equation}
where the matrices $S_{11} \in \mathbb{R}^{n_{\Gamma} \times
n_{\Gamma}}$, $S_{12} \in \mathbb{R}^{n_{\Gamma} \times N}$ and $S_{22}
\in \mathbb{R}^{N \times N}$. The matrix $S_{22}$ is negative definite
and has a nearly-diagonal
structure with reducing entries as the Newton iteration progresses, while $S_{12}$ can be computed cheaply in parallel. The
main focus will therefore be the matrix $S_{11}$ which is associated
with the interface displacement nodes. This block dominates the Schur complement for the Jacobian,
and so the aim is to design a preconditioning procedure based on the
structure of $S$ above and in a manner that provides a suitable approximation $\tilde{S}_{11}$
of $S_{11}$. 
We note here that one can view the block-structure (\ref{schur}) as
the discretization of an operator corresponding to $S_{11}$ and
constrained by conditions incorporated in the remaining blocks.
\begin{remark}
  The Schur complement is symmetric and indefinite, with a negative
  inertia equal to $i_{-}=N+1$, where $N$ is the number of
  subdomains. Standard optimization approaches, such as projected CG,
  cannot be used in this case. Similarly, we expect positive
  definite preconditioners to be less effective. Consequently, we
  devised a novel approach which incorporates the
  indefiniteness (and the constraints) into our preconditioner. The resulting method can be
  loosely described as being of constrained type (cf. \cite{gould}).
\end{remark}

Given the structure of $S$ we propose a preconditioner with a similar
block form but with $S_{11}$ replaced by a suitable approximation. A direct calculation using the block form (\ref{ExpandedJ}) of
the Jacobian matrix yields
\begin{equation}
S_{11}=S_{\Gamma\Gamma}+E\label{eqE}
\end{equation}

where
\begin{equation}
  \label{elastschur}
  S_{\Gamma\Gamma}{ \mathrel{\,\vcenter{\baselineskip0.5ex \lineskiplimit0pt\hbox{\small.}\hbox{\small.}}}=\, }A_{\Gamma\Gamma}-A_{\Gamma I}A_{II}^{-1}A_{I\Gamma}
\end{equation}
is the Schur complement arising in a Dirichlet-Dirichlet
non-overlapping domain decomposition method for the elasticity
equations. The above splitting of $S_{11}$ suggests the following candidate for
a preconditioner:
\begin{equation}
S_0=
\begin{bmatrix}
{S_{\Gamma\Gamma}} & {S_{12}} & {\bf 0} \\
{S_{12}^T} & {S_{22}} & {\mathbf{1}_N} \\
{\bf 0}^T & {\mathbf{1}^T_N}  & 0
\end{bmatrix}.\label{schur0}
\end{equation}
This choice is computationally expensive due to the dense matrix
$S_{\Gamma\Gamma}$ and $S_0$ can only be seen as an ideal preconditioner,
very much like the Schur complement itself. However, the (1-1)-block
can afford approximations which make the application of $S_0$ efficient.
In the following, we restrict our
attention to $S_{\Gamma\Gamma}$. This matrix is the finite element representation of the
Steklov-Poincar\'{e} operator corresponding to the interface problem
for the elasticity equations, which is known to be continuous and
coercive on $\Lambda^2$.
The above properties can be shown using the continuity and
coercivity of the bilinear form $a_\rho(\cdot,\cdot)$ on $V$; for
details, see \cite{Quarteroni99}. The restriction to $\V_h$, in the
context of the finite element discretization of our problem, preserves these
properties; in turn, $S_{\Gamma\Gamma}$ can be shown to be spectrally equivalent
to the matrix representation of a norm on $\Lambda_h^2$. We describe
this discrete norm below.
\subsection{Discrete fractional Sobolev norms}
In order to describe the relevant norms further, we first describe the
relevant function spaces defined on the interface $\Gamma$. Let $\nabla_{\Gamma}$ represent the tangential gradient of a scalar function $v(\mathbf{x})$ such that
\[
\nabla_{\Gamma} v(\mathbf{x}) = \nabla v(\mathbf{x}) - \mathbf{n} \left( \mathbf{n} \cdot \nabla v(\mathbf{x}) \right),
\]
corresponding to the projection of the gradient of $v$ onto the plane
tangent to $\Gamma$ at the point $\mathbf{x} \in \Gamma$. We now define
\[
H^1(\Gamma) { \mathrel{\,\vcenter{\baselineskip0.5ex \lineskiplimit0pt\hbox{\small.}\hbox{\small.}}}=\, } \left\{ v \in L^2(\Gamma) \hspace{0.05in} \bigg{|} \hspace{0.05in} \displaystyle\int_{\Gamma} \lvert \nabla_{\Gamma} v \rvert^2 \,ds(\Gamma) < \infty \right\}.
\]
Let $\gamma:=\Gamma \cap \partial \Omega_D$ denote the set of
points on the Dirichlet boundary of the domain which are also on the
interface $\Gamma$; when this set is non-empty, we define the space 
\[
H^1_D(\Gamma) := \left\{ v \in H^1 (\Gamma) \hspace{0.05in} {\vert_{}} \hspace{0.05in} v_{|_\gamma} = 0 \right\}.
\]
We recall here that the space $\Lambda$ introduced above is the
fractional Sobolev space of index 1/2. More generally, we define the
scale of spaces
$\Lambda_\theta$ as the interpolation spaces of index $\theta\in(0,1)$
corresponding to the pair of spaces $\left\{H^1_{0}(\Gamma),L^2(\Gamma)\right\}$ in the
sense of Lions and Magenes \cite{Lions70}:
$$\Lambda_\theta{ \mathrel{\,\vcenter{\baselineskip0.5ex \lineskiplimit0pt\hbox{\small.}\hbox{\small.}}}= \, }[H^1_\gamma(\Gamma), L^2(\Gamma)]_\theta.$$
Using a formulation involving discrete interpolation spaces with
application to finite element spaces, one can define a finite
dimensional space $\Lambda_{\theta,h}\subset \Lambda_\theta$, together
with the following matrix representation of a discrete interpolation
norm \cite{ariolikourounisloghin}
\begin{equation}
  \label{honehalf}
  H_\theta{ \mathrel{\,\vcenter{\baselineskip0.5ex \lineskiplimit0pt\hbox{\small.}\hbox{\small.}}}= \, }[L_\Gamma,M_\Gamma]_{\theta}{ \mathrel{\,\vcenter{\baselineskip0.5ex \lineskiplimit0pt\hbox{\small.}\hbox{\small.}}}= \, }M_\Gamma(M_\Gamma^{-1}L_\Gamma)^{1-\theta},
\end{equation}
where $M_\Gamma$ and $L_\Gamma$ are the mass and Laplacian matrices assembled on $\Gamma$
using the restriction of the finite element basis for $\V_h$ to the
interface. Using this norm representation, the continuity and
coercivity properties of the elasticity Schur complement
$S_{\Gamma\Gamma}$  translate into the following spectral equivalence
\cite[p. 129]{Turner14}
\[
  c_1\norm{\u}{H}\leq \norm{\u}{S_{\Gamma\Gamma}}\leq c_2\norm{\u}{H},
\]
where $H{ \mathrel{\,\vcenter{\baselineskip0.5ex \lineskiplimit0pt\hbox{\small.}\hbox{\small.}}}=\, }H_{1/2}\oplus H_{1/2}$. This equivalence is the basis of our candidate for the
approximation of the Schur complement $S$.

The matrix $H_\theta$ is in general full and expensive to
compute. However, an implementation of the action of the inverse of
$H_\theta$ on a given vector can be achieved cheaply using a Krylov
subspace approximation constructed via a generalised inverse Lanczos
iteration which only requires the application of the inverse of the
sparse matrix $L_\Gamma$ for a small number of steps
\cite{ariolikourounisloghin}. In the following section we indicate
how to extend this procedure to the case of our proposed preconditioner.

\subsection{Constraint preconditioners}
Given the constraint form (\ref{schur}) of the interface Schur
complement, we propose the following choice of preconditioner which
preserves the constraining blocks while replacing the (1,1)-block by a
spectrally equivalent matrix:
\begin{equation}
\label{S1}
S_1{ \mathrel{\,\vcenter{\baselineskip0.5ex \lineskiplimit0pt\hbox{\small.}\hbox{\small.}}}=\, }
\begin{bmatrix}
H & {S_{12}} & {\bf 0} \\
{S_{12}^T} & {S_{22}} & {\mathbf{1}_N} \\
{\bf 0}^T & {\mathbf{1}^T_N}  & 0
\end{bmatrix}.
\end{equation}
The notion of constraint preconditioning has been extensively analysed
in the context of saddle-point problems \cite{gould, luksan}. In our case,
the matrix structure is of a different nature; however, given the
spectral equivalence between $S_{11}$ and $H$, we hope to achieve a
useful similarity  between $S$ and $\tilde{S}$. Our motivation is the
following result.
\begin{proposition}\label{eigprop}
  Consider the generalized eigenvalue problem
  $${\cal K}\z=\lambda{\cal G}\z,$$
  where
  $${\cal K}=\begin{bmatrix}
K & D & {\bf 0} \\
{D^T} & {F} & {\mathbf{1}_N} \\
{\bf 0}^T & {\mathbf{1}^T_N}  & 0
\end{bmatrix},~~~{\cal G}=\begin{bmatrix}
G & D & {\bf 0} \\
{D^T} & {F} & {\mathbf{1}_N} \\
{\bf 0}^T & {\mathbf{1}^T_N}  & 0
\end{bmatrix},$$
with $K, G\in\mathbb{R}^{n_\Gamma\times n_\Gamma},
F\in\mathbb{R}^{N\times N}$ nonsingular. Let $Z$ be a basis for the nullspace of $\mathbf{1}_N$.
Then
\begin{enumerate}
\item $\lambda=1$ with multiplicity $N+1$;
\item the remaining $n_\Gamma$ eigenvalues satisfy the eigenvalue
  problem
  $$(K-Q)\z=\lambda(G-Q)\z$$
  where $Q=DZ(Z^TFZ)^{-1}(DZ)^T$.
\end{enumerate}
\end{proposition}
\begin{proof}
Similar to proof given in \cite{gould} (see also \cite[pp 180--181]{Turner14}).
\end{proof}
\begin{remark}
  The above result indicates that the increase in size by $N$ due to our problem reformulation
  (\ref{DWFTOlog1}--\ref{massconstraintkb}) is automatically taken care
  of by our constraint interface preconditioner, which maps $N+1$
  eigenvalues to 1. The remaining eigenvalues will depend on the
  closeness of our preconditioner $H$ to $S_{11}$. As described above,
  we tried to incorporate this property by a choice of $H$ that is
  spectrally equivalent to the dominant part of $S_{11}$, namely $S_{\Gamma\Gamma}$.
\end{remark}

The main issue with the definition of $S_1$ is the full matrix $H$
arising in the (1,1)-block. This matrix needs to be computed via an
expensive matrix square-root calculation; moreover, it needs to be
employed in order to implement the action of the inverse of $S_1$ on a
given vector. A practical alternative can be derived from the following
constrained Lanczos decomposition of $S_1$.

Let $$V_\Gamma^TL_\Gamma V_\Gamma=T_\Gamma, ~~V_\Gamma^TM_\Gamma V_\Gamma=I_\Gamma$$ denote the generalized
Lanczos decomposition of the pencil $[L_\Gamma,M_\Gamma]$ in exact
arithmetic, where $V_\Gamma$ is an orthogonal matrix and $T_\Gamma$ is
a tridiagonal, symmetric and positive-definite matrix. Define
$T=T_\Gamma\oplus T_\Gamma, V=V_\Gamma\oplus V_\Gamma$. Then
$H=VT^{1/2}V^T$ (see also \cite{ariolikourounisloghin}). Consider the QR-factorization
$$UR=
\begin{bmatrix}
  {S_{12}} & {\bf 0}
\end{bmatrix} V,$$
where $U$ is orthogonal and define
$$D:=
U^T\begin{bmatrix}
  S_{22}& {\bf 1}\\
  {\bf 1}^T & 0
\end{bmatrix}U.
$$
We obtain the following orthogonal factorization of $S_1$
$$S_1=
\begin{bmatrix}
  V & \\ & U
\end{bmatrix}
\begin{bmatrix}
  T^{1/2}  & R^T\\R & D
\end{bmatrix}
\begin{bmatrix}
  V^T& \\ & U^T
\end{bmatrix}=:W{\cal T}W^T.
$$
We will refer to the above representation of $S_1$, seen as a
two-by-two block matrix, as the constrained
Lanczos factorization of $S_1$.\\
Given our preconditioning task, consider now the product
$$\z=S_1^{-1}{\bf v}=W{\cal T}^{-1}W^T{\bf v}.$$
An approximation to $\z$ can be constructed using a partial
factorization:
$$\z\approx \z_k:= W_k{\cal T}_k^{-1}W_k^T{\bf v}$$
where we define
$$
W_k:=
\begin{bmatrix}
  V_k&\\&U_k
\end{bmatrix},~~~
{\cal T}_k:=
\begin{bmatrix}
  T_k^{1/2}  & R_k^T\\R_k & D_k
\end{bmatrix},~~~
D_k=U_k^T\begin{bmatrix}
  S_{22}& {\bf 1}\\
  {\bf 1}^T & 0
\end{bmatrix}U_k.
$$
The preconditioning operator implicit in the definition of
$\z_k$ will be denoted by $S_2$:
$$\z\approx S_2^{-1}{\bf v}.$$
Its implementation requires:
\begin{itemize}
\item[i.] a partial Lanczos factorization of the (1,1)-block of $S_1$; this yields a
  matrix $V_k$ of Lanczos vectors and a block tridiagonal matrix $T_k$;
\item[ii.] a QR-factorization of the (1,2)-block of $S_1$ multiplied by $V_k$; this
  yields the factors $U_k,R_k$ arising in $W_k$ and ${\cal T}_k$, respectively.
\end{itemize}
The number of Lanczos vectors in $V_k$ is expected to be small;
correspondingly, the sizes of $R_k, D_k$ will also be small and the
resulting matrix ${\cal T}_k$ will be easy to invert. The overall
complexity for the above procedure is of order $O(kn_\Gamma)$. Thus,
the application of the interface preconditioner will not dominate the
cost of a subdomain solve provided we work with 
subdivisions for which $kn_\Gamma=O(n_{I_i}^2)$, where we assumed that the
cost of inverting a subdomain Jacobian matrix is the cost that a direct method
requires to invert a banded matrix of size $n_{I_i}$ with bandwidth $\sqrt{n_{I_i}}$.
\section{Numerical experiments}
In order to illustrate the performance of the solution method
described in Section \ref{IPAlg}, we consider a model test problem
based on compliance design. The problem involves a cantilever beam posed
on a rectangular design domain as illustrated in Figure
\ref{DomainLayouts}, with clamping applied along the left edge and
a force applied in the middle of the right edge. The density contour plot
corresponding to the optimal design is displayed in Figure
\ref{SolnPlots}. Symmetry is  a feature which may be exploited computationally; however in this paper we chose to retain the original design domain in order to test the performance of our solution method on the full problem.
\begin{figure}[!b]
  \begin{center}
    \includegraphics[height=2.4cm]{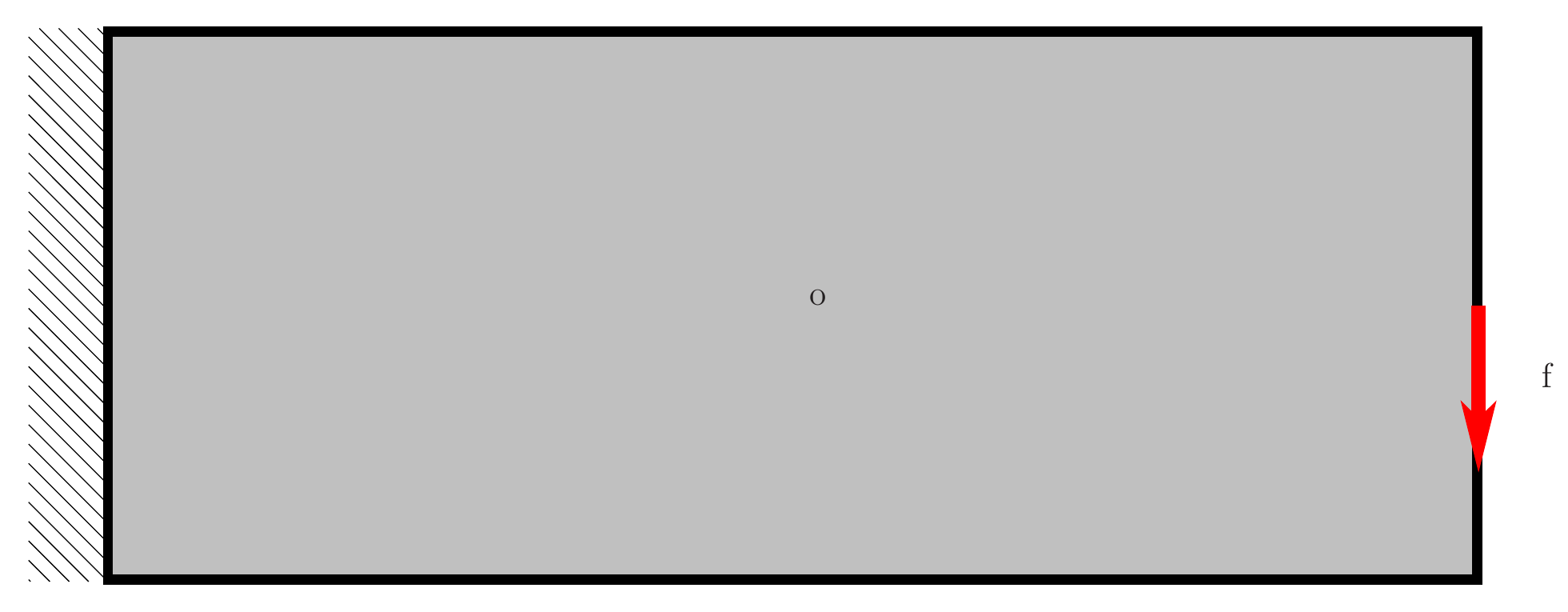}
  \end{center}
\caption{Illustration of the Cantilever Beam problem.}\label{DomainLayouts}
\end{figure}
\begin{figure}[!t]
  \begin{center}
    \includegraphics[width=5.4cm]{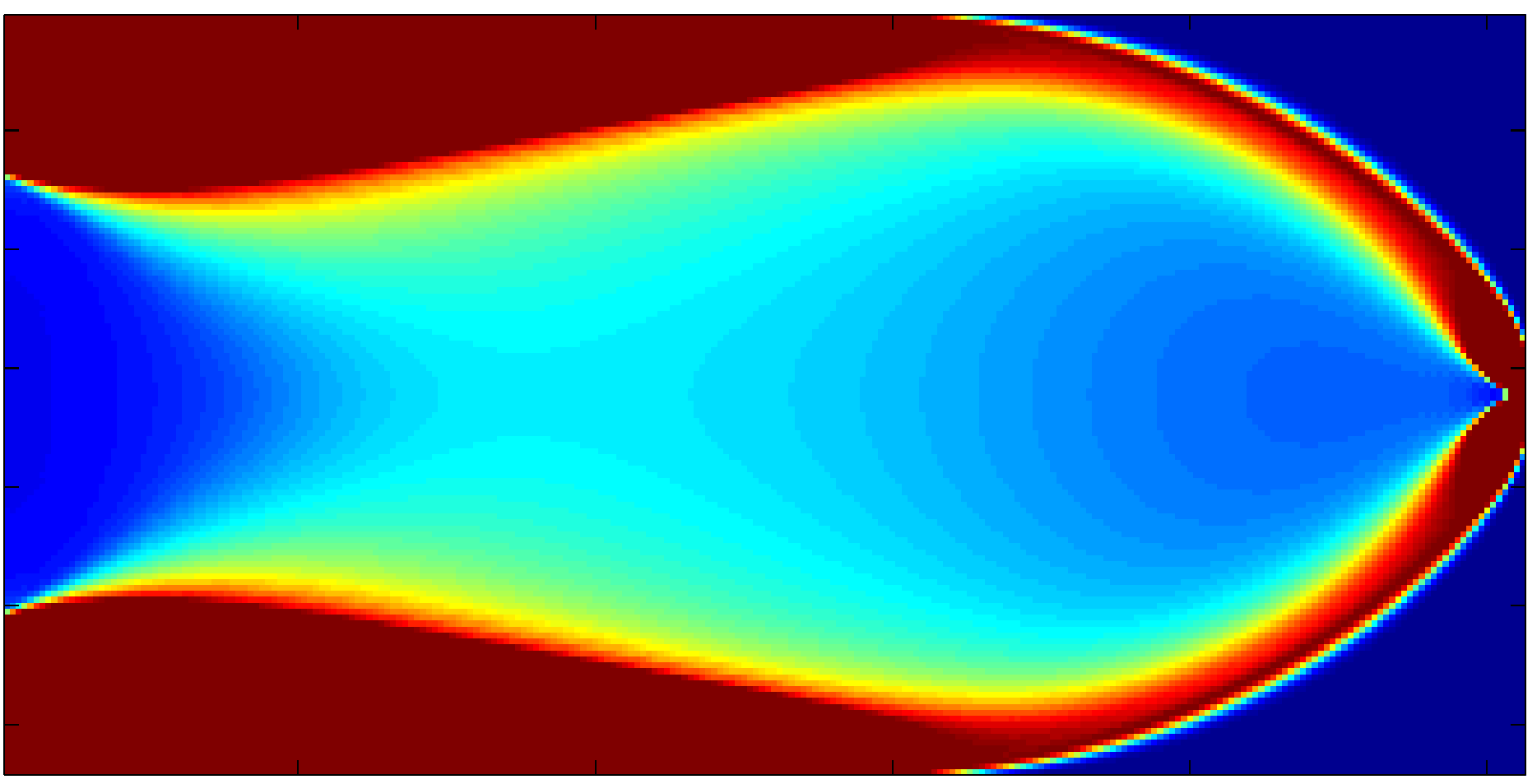}
  \end{center}
\caption{Contour plots of optimal density distribution for the model problem.} \label{SolnPlots}
\end{figure}
\subsection{Implementation details}\label{IPAlg}
\subsubsection{Finite Element Method} We use a subdivsion ${\cal T}_h$
of $\Omega$ into square elements of size $h$. We seek approximations
$$u_h\in \V_h=\left\{v\in (H_D^1(\Omega))^2:v_{\mid T}\in {\cal
    P}_2(T)\right\}, ~~\rho_h\in \Q_h=\left\{q\in L^\infty(\Omega):q_{\mid T}\in {\cal
    P}_0(T)\right\},$$
for all $T\in{\cal T}_h$, where ${\cal P}_k(U)$ denotes the space of
degree $k$ polynomials defined on $U$.
We employ uniform refinement, in order to exhibit the behaviour of our
preconditioning technique with
respect to the mesh parameter $h$. We note here that the space $\Q_h^c$
is not constructed explicitly, and that it was only introduced in
order to provide the mathematical description of our problem.
\subsubsection{Interior Point Algorithm} We use a text-book interior-point algorithm; the
key details are included below; for full details see \cite[Ch. 19]{nocedal-wright}.

We repeat the following steps until convergence:
\begin{enumerate}
\item Solve a sequence of Newton systems (\ref{NK1}) to get $\Delta \y^k$.
\item Find the step length $\alpha$ (see below).
\item Update the solution
$$
  \y^k = \y^{k-1} + \alpha \Delta \y^k\,.
$$
\item Update the penalty parameters via
$$
  r:=r/4,~~~s:=s/4.
$$
\end{enumerate}
We start with $r=s=1$ and continue until they are brought below a
tolerance of $10^{-6}$. A more sophisticated version of the code, with
an adaptive choice of the penalty parameters $r$ and $s$ can be found in
\cite{Jarre98}. However, we have never observed any difficulties with
this simple update of $r$ and $s$.

\vspace{.1cm}\noindent {\bf Step length.} In the interior-point method
we cannot take a full Newton step resulting from the solution of (\ref{NK1}), as this would lead to an
infeasible point. We propose a technique that will
effectively involve finding $\alpha_L$ such that $\rho_i + \Delta
\rho_i \geq \underline{\rho}$ for $i$ such that $\Delta \rho_i < 0 $
and $\alpha_U$ such that $\rho_i + \Delta \rho_i \leq \overline{\rho}$
for $i$ such that $\Delta \rho_i > 0 $. Therefore, we consider
obtaining $\alpha_L$ and $\alpha_U$ as follows
\begin{align*}
\alpha_L &= 0.9 \cdot \min_{i:\Delta\rho_i < 0} \left\{ { \frac{ \underline{\rho} - \rho_i}{\Delta \rho_i} } \right\}, \\
\alpha_U &= 0.9 \cdot \min_{i:\Delta\rho_i > 0} \left\{{ \frac{ \overline{\rho} - \rho_i}{\Delta \rho_i} } \right\}. 
\end{align*}
The constant $0.9$ represents an appropriate shortening of the Newton
step to the interior of the feasible region. In the event
that both $\alpha_L$ and $\alpha_U$ are greater than $1$, the step will
be shortened appropriately:
\begin{equation*}
\alpha = \min \left\{{ \alpha_L , \alpha_U , 1  } \right\}.
\end{equation*}
This procedure is relatively simple
and straightforward to both implement. A more sophisticated
line search procedure is described in \cite{Jarre98} and could
potentially be used here. However, it will be illustrated later that
the current technique was able to yield desirable results.

\vspace{.1cm}
\noindent{\bf Initial approximation.}
We start the interior-point algorithm with a uniform distribution of
the design variable $\pmb{\rho}$ chosen with respect to the
mass
constraint (\ref{massconstraint1})
$$
  \boldsymbol{\rho}_{\rm ini}=\mathbf{1} \,.
$$
The initial displacements $\mathbf{u}_{\rm ini}$  are then computed from the equilibrium
equation $$A\left(\boldsymbol{\rho}_{\rm
ini}\right)\mathbf{u}_{\rm ini} = \mathbf{f}.$$
\subsubsection{Domain decomposition}
We used only regular subdivisions into rectangular subdomains. The
corresponding sizes of the resulting interfaces are illustrated in
Table \ref{meshtable}.  It is evident (and well-known) that in order to balance the
complexities of the interface and subdomain problems, the increase in
the number of subdomains should be paralleled by a decrease in $h$. 
Aside from regular decompositions, one could also decompose
the domain in an adaptive fashion based on the changing nature of the
design. This could be carried out either at each outer iteration, or
alternatively once after a fixed number of outer iterations based on
the (previously mentioned) observation that the density will only be
subject to minor changes after a relatively small number of iterative
steps. In terms of a non-regular subdivision, the graph partitioning
tool METIS \cite{Karypis98} may be used in order to partition the
finite element mesh into non-regular subdomains.
\begin{table}[!t]
 \renewcommand{\arraystretch}{1.3}
 \small{
   \hspace*{0.0cm}
   \begin{center}
     \begin{tabular}{|c|c|c|c|c|}
       \cline{3-5}
       \multicolumn{2}{c}{} &\multicolumn{3}{|c|} {$n_\Gamma$} \\
       \cline{1-5}
       $h$ & $n$& $N=4$  & $N=16$& $N=64$ \\
       \hline
       $1/64$   & 41,355 & 384 &1,140 &2,604 \\
       $1/128$   & 164,619 &768 &2,292 & 5,292\\
       $1/256$   & 657,027 &1,536 &4,584 & 10,668\\
       \hline
     \end{tabular}
     \caption{Mesh information for Cantilever Beam experiment}\label{meshtable}
   \end{center}}
\end{table}
\subsubsection{GMRES} We used an inexact Newton-GMRES method
preconditioned by either $S_0, S_1$;  flexible GMRES was used for the case of $S_2$. The
GMRES stopping criterion was the reduction of the norm of the initial residual
by a factor of $10^6$.
\subsubsection{Constrained Lanczos factorization} The implementation
of the action of $S_2^{-1}$ requires first to generate the matrices
$S_{12}$ and $S_{22}$. This is achieved as an additional
pre-processing step involving one set of subdomain solves. The number
of Lanczos vectors is taken to be $k=O(\sqrt{n_\Gamma})$, so that the
overall complexity of using the preconditioner is of order $O(n_\Gamma^{3/2})$.
\subsection{Numerical results}
Table \ref{Results} displays the results for our test case for a range
of mesh and subdomain sizes. The results were obtained using a Linux
machine with an Intel{\tiny{\textregistered}} Core\texttrademark \, i7
CPU 870 $@$ 2.93 GHz with 8 cores. The upper and lower limits on the density
$\rho$ were set at $1$ and $10^{-2}$ respectively, with the
permissible volume in each test case defined to be ${\cal M}_\Omega/2$.

 The first observation arising from our numerical experiments is that
 $S_0$ has performance independent of mesh-size and almost independent
 of the number of subdomains; moreover, the number of iterations is
 low (averages between 5--8 iterations on the largest problem), with
 only a small departure from the optimal count of 2 iterations corresponding to the case
 where the exact Schur complement is used. This confirms the earlier 
 assumption that the matrix $E$ in (\ref{eqE}) is negligible and that
 the preconditioner $S_0$ is optimal in the same sense as the Schur
 complement. 
 \begin{table}[t]
 \begin{center}
 {\centering{
 {\resizebox{\textwidth}{!}{
 \def\arraystretch{1.5}
 \begin{tabular}{| c | c || c  c  c  | c  c  c  | }
 \cline{3-8}
 \multicolumn{1}{c}{} & \multicolumn{1}{c|}{} &
                                                \multicolumn{3}{c|}{Avg
                                                GMRES (Newton)} & \multicolumn{3}{c|}{Total GMRES Its.} \\
\hline
 \multicolumn{2}{|c|}{No. Subdomains:} & 4 & 16 & 64 & 4 & 16 & 64\\
\hline 
Preconditioner \, & \backslashbox{ \hspace{0.2in} $ h$ }{\vspace{-0.08in} $\theta$ } & 0.5 & 0.6 & 0.7  & 0.5 & 0.6 & 0.7 \\
 \hline
 \hline
 \multirow{3}{*}{$S_0$} 
 &1/64 &7.29 (14) &10.57 (14) &13.86 (14) &102 &148 & 194\\
 &1/128 &5.93 (15) &8.13 (15) &8.93 (15) &89 &122 &134 \\
 &1/256 &5.25 (16) &6.81 (16)& 7.50 (16)&84  & 109& 120\\
\hline
\multirow{3}{*}{$S_1$} 
 &1/64 &11.36 (14) &25.64 (14) &35.00 (13)&159 &359 & 455\\
 &1/128 &10.60 (15) &21.80 (15) &33.12 (16) & 159 &327 &530\\
 &1/256 &10.25 (16) &19.94 (16) &29.82 (17) &164 &319 & 507\\
\hline
 \multirow{3}{*}{$S_2$} 
 &1/64 &12.21 (14) &31.36 (14) &50.46 (13)&171 &439 &656\\
 &1/128 &12.00 (15) &27.47 (15) &43.47 (15) & 180&412 &652\\
 &1/256 &11.00(16) &26.41 (17) &40.06 (17)&176 &449 & 681\\
 \hline
 \end{tabular}
 }}}}
 \caption{GMRES (Newton) iterations required for solving (\ref{DWFTOlog}) using preconditioners $S_0, S_1, S_2$.}
 \label{Results}
 \end{center}
 \end{table}

With regard to parameter dependence, we note that the number of
iterations decreases with $h$ in all experiments. This is somewhat
expected, given that $S_2$ is essentially an approximate
implementation of $S_1$, which incorporates in the (1-1) block the
finite element discretization of the continuous Steklov-Poincar\'{e}
operator for the elasticity equations. An interesting fact is
that the preconditioning technique $S_2$ appears to outperform occasionally
 the preconditioner $S_1$; however, different preconditioners lead to
 different Newton convergence histories given the adaptive stopping
 criterion employed, which may result in overall complexities more
 favourable for the constrained preconditioning approach.

The table also indicates what
appears to be a logarithmic dependence on the number of subdomains. We
found this difficult to analyze, but this behaviour is not unlike that
exhibited by similar substructuring preconditioning techniques in
\cite{ariolikourounisloghin}. This suggests that the preconditioner
$S_2$ has the ability to match the properties of the
discrete fractional Sobolev norm in a constrained setting. This
represents a novel approach which could be useful in other constrained
optimization settings and for other PDE models. 

The values of $\theta$ listed in Table \ref{Results} were chosen by
experimentation based on similar observations for the linear
elasticity problem reported in \cite{Turner13}, where it is noted that
different values of $\theta$ may be able to provide a closer
approximation to the decay of the associated Steklov-Poincar\'e
operator. Similar findings were found numerically in this work also,
with the best values of $\theta$ used in order to produce the results
reported in the table.

Despite the fact that a logarithmic dependence is noted for an
increasing number of subdomains, the computational benefits gained as
a result of distributing calculations amongst an increasing number of
processors can lead to a significant speedup when compared to solving
the original problem on the global domain $\Omega$. Exact calculations
displaying this behaviour are not included here, however the
interested reader is referred to \cite{Turner13}, where notable
speedup was observed under the preconditioner described in
(\ref{honehalf}) for the Optimality Criteria method for topology optimization, which, unlike the
interior point approach in this paper, requires the solution of a
linear elasticity problem at each step.

\section{Conclusions} We described a novel domain decomposition approach
coupled with an interior point method for compliance
minimisation problems arising in the field of topology
optimization. The problem was reformulated in order to allow for
well-posed subdomain problems which could be viewed as local Jacobian
solves. This was an important step which allowed for the domain
decomposition method to be well-defined. The resulting interface Schur
complement problem yielded an indefinite matrix which included a
global volume constraint; consequently, an indefinite preconditioner
was devised in order to incorporate the properties and the structure
of the interface Schur complement matrix. This resulted in a technique
requiring the inversion of a so-called constrained discrete fractional
Sobolev norm, the application of which was performed via a certain
constrained Lanczos procedure. We tested the resulting method on a standard
test problem of topology optimization, with experiments indicating independence on the mesh
parameter, although a depedence on the number of subdomain was
noticed. 

Some of our current investigations include  non-regular decompositions, as well as adaptive
decompositions based on current iterates. The constrained Lanczos
factorization, which allowed the sparse implementation of our interface
preconditioner, will be the subject of further study. The precise role played by the parameter
$\theta$ also requires more analysis and experimentation in order to
quantify the most appropriate value of $\theta$ for a decomposition into any given
number of subdomains. Finally, the reformulation of the problem
resulting in Jacobian sub-problems on subdomains is also worthy of
further investigation as it points to nonlinear domain decomposition approaches, such as that described in \cite{Turner13a}.

\bibliographystyle{siam}
\bibliography{ddnkto}

\end{document}